\documentclass[12pt,a4paper,russian,titlepage]{article}
\usepackage{babel,amsmath,amssymb,amsthm,graphicx,amsfonts,latexsym,epsfig,longtable,wrapfig}

\usepackage{xypic}

\newcommand{\range}[2]{
    \{#1, \ldots, #2 \}
}

\newcommand{\E}{\exists}

\newtheorem{proposition}{Утверждение}
\newtheorem{theorem}{Теорема}

\begin{document}

\begin{flushright}
УДК 510.67
\end{flushright}

\begin{center}
{\Large
Логическая сложность свойства наличия индуцированного подграфа изоморфного данному для некоторых семейств графов\footnote{Настоящая работа выполнена при финансовой поддержке гранта РНФ 18-71-00069.}
}

\vspace{0.5cm}

{\large

М.Е. Жуковский\footnote{Московский физико-технический институт, лаборатория продвинутой комбинаторики и сетевых приложений}, Е.Д. Кудрявцев, М.В. Макаров, А.С. Шлычкова
}
\vspace{0.5cm}

    Аннотация\\
    
    \end{center}
    
Доказано, что для любого $\ell\geq 4$ найдется граф на $\ell$ вершинах и формула первого порядка кванторной глубины не более $\ell-1$, записывающая свойство содержать индуцированный подграф, изоморфный этому графу. Доказано, что формула первого порядка, записывающая свойство содержать индуцированный подграф на $\ell$ вершинах, изоморфный фиксированному дизъюнктному объединению изоморфных полных многодольных графов, имеет глубину по крайней мере $\ell$. Наконец, доказано, что для любого графа на $\ell\leq 5$ вершинах формула первого порядка, записывающая свойство содержать индуцированный подграф изоморфный заданному, имеет глубину по крайней мере $\ell-1$.
    
    \vspace{0.5cm}

Ключевые слова: формула первого порядка, кванторная глубина, индуцированный подграф, логическая сложность.


    \section{Введение}
    
    Рассмотрим логику первого порядка \cite{Veresh,Survey,Libkin2004} $\mathcal{F}$ над множеством конечных графов с сигнатурой, состоящей из двух предикатных символов $\sim, =$, обозначающих смежность вершин и равенство соответственно (оба предиката имеют арность 2).
    
    Под {\it свойствами графов} мы будем подразумевать такие множества графов, что любой класс изоморфизма либо принадлежит множеству целиком, либо вообще его не пересекает. Иными словами, если граф $G$ обладает некоторым свойством, то и любой изоморфный ему граф тоже обладает этим свойством.
    
    Будем говорить, что предложение $\varphi$ {\it выражает} свойство графов $\mathcal{C}$, если истинность предложения равносильна обладанию свойством. Иными словами, для любого графа $G$ выполнено
    $$
     G\models\varphi\Leftrightarrow G\in\mathcal{C}.
    $$
    
    В этой работе мы будем рассматривать свойство графов $\mathcal{C}[F]$ {\it содержать индуцированный подграф, изоморфный данному графу $F$} (граф $H$ называется {\it индуцированным подграфом} графа $G$, если множество вершин графа $H$ является подмножеством множества вершин графа $G$, а множество ребер индуцировано этим множеством вершин, т.е. любые две вершины $u,v$ графа $H$ смежны в $H$ тогда и только тогда, когда они смежны в $G$). Такие свойства представляют большой интерес как c теоретической точки зрения, так и для различных приложений. Отметим отдельно, что если бы свойство $\mathcal{C}[K_{\ell}]$ ($K_{\ell}$ --- это полный граф на $\ell$ вершинах) проверялось бы за $n^{o(\ell)}$ для $n$-вершинных графов, то неверна была бы гипотеза об экспоненциальном времени~\cite{Hypothe}.
    
    Наш интерес к рассмотрению свойств $\mathcal{C}[F]$ в контексте логической сложности обусловлен именно временной сложностью проверки этого свойства. Более формально, для графа $F$ обозначим через $D[F]$ минимальную {\it кванторную глубину} (под кванторной глубиной формулы подразумевается наибольшая длина последовательности вложенных кванторов в этой формуле; строгое определение см., например, в~\cite{Survey,Libkin2004}) предложения из $\mathcal{F}$, выражающего свойство $\mathcal{C}[F]$. Например, формула
    $$
    \E x_1 \E x_2 \E x_3\quad( x_1 \sim x_2 \land x_2 \sim x_3 \land x_1 \nsim x_3 \land x_1\neq x_3)
    $$
    записывает свойство содержать подграф, изоморфный $P_3$ (здесь и далее $P_{\ell}$ --- это простой путь на $\ell$ вершинах), и имеет глубину 3. Поэтому $D[P_3] \leq 3$. Очевидно, для любого графа $F$ на $\ell$ вершинах выполнено $D[F] \leq \ell$. Так как истинность формулы глубины $r$ на $n$-вершинном графе можно проверить за $O(n^r)$~\cite{Libkin2004}, то и определить принадлежность $n$-вершинного графа множеству $\mathcal{C}[F]$ можно за $O(n^{D[F]})$. Этот факт мотивирует интерес к нахождению величины $D[F]$ для произвольного графа $F$.\\
    
    Заметим, что, насколько нам известно, наилучшей верхней оценкой на сложность проверки $\mathcal{C}[F]$ в общем случае (для произвольного графа $F$ на $\ell$ вершинах) является $O\left(n^{\omega\ell/3+O(1)}\right)$~\cite{NesetrilP85,Eisenbrand}, где $\omega\in[2,2.373)$ --- экспонента быстрого перемножения матриц~\cite{Exp_Fast_Matrix}. Заметим, что $O(1)$ в показателе, разумеется, не зависит от $\ell$. Возникает естественный вопрос, а существуют ли графы, для которых $D[F]<\omega\ell/3$? К сожалению, этот вопрос остается открытым. Возможность положительного ответа обусловлена наилучшей известной нижней оценкой на $D[F]$ в общем случае~\cite{VerbitskyZ17}: 
    \begin{equation}
    D[F]\geq\max\left\{\left\lfloor\frac{1}{2}\ell-2\log_2\ell+3\right\rfloor,\chi(F),\frac{e(F)}{v(F)}+2\right\},
    \label{general_lower}
    \end{equation}
    где $\chi(F)$ --- хроматическое число графа $F$, $v(F)=\ell$ --- количество вершин графа $F$, а $e(F)$ --- количество ребер графа $F$. Действительно, легко видеть, что максимум из этих трех величин равен первой для почти всех графов (о хроматическом числе случайного графа можно почитать, например, в~\cite{Janson_RG}).
    
    Тем не менее, до настоящей работы был известен пример лишь одного графа (с точностью до реберного дополнения), для которого $D[F]<v(F)$. Так, из работы~\cite{VerbitskyZ17} известно, что для всех графов $F$ с $v(F)\leq 4$ выполнено $D[F]=v(F)$ тогда и только тогда, когда граф $F$ отличен от diamond-graph (т.е. треугольник с дополнительным ребром, иисходящим из одной из вершин треугольника к четвертой вершине) и от его реберного дополнения (т.е. от дизъюнктного объединения вершины и $P_3$). Для этих двух ``исключительных'' графов на 4 вершинах $D[F]=3$.
    
    Заметим, наконец, что из~(\ref{general_lower}), в частности, следует, что $D[K_{\ell}]=\ell$, так как $\chi(K_{\ell})=\ell$.
    
    Подробнее об истории этой задачи и о других смежных результатах можно прочитать, например, в \cite{VerbitskyZ17,VerbitskyZTOCS,ZDAN}.\\
    
    В настоящей работе мы, во-первых, доказали, что существуют графы любого размера, не меньшего 4, для которых $D[F]\leq v(F)-1$. Более того, таких графов достаточно много: подобный граф $F$ можно получить из произвольного графа, добавив к нему дизъюнктно вершину и $P_3$. Кроме того, мы значительно расширили класс графов, для которых $D[F]=v(F)$. Наконец, мы установили, что для всех графов на 5 вершинах $D[F]\geq 4.$ Поэтому, если и существует граф $F$ с $D[F]\leq v(F)-2$, то количество вершин такого графа не меньше 6.
    
    Формулировки описанных результатов и некоторые обсуждения приведены в разделе~\ref{results}. В разделе~\ref{ehren} описан основной инструмент получения нижних оценок величины $D[F]$, который мы использовали. Разделы~\ref{proof1},~\ref{proof2},~\ref{proof3} содержат доказательства результатов. А раздел~\ref{discussions} посвящен обсуждению открытых вопросов.
    
    \section{Основные результаты}
    \label{results}
    
    Ниже $A\sqcup B$ обозначает дизъюнктное объединение графов $A$ и $B$, $mA$ обозначает дизъюнктное объединение $m$ изоморфных копий графа $A$, а $K_1$ --- граф без ребер, состоящий из всего одной вершины.
    
    \begin{theorem}
        \label{t1_1}
        Пусть $H$ --- произвольный граф, $F= P_3\sqcup K_1\sqcup H$.
        Тогда $D[F] \leq v(F)-1$.
    \end{theorem}
    
    К сожалению, нам не удалось доказать, что эту оценку нельзя улучшить ни для каких графов $F$ обозначенного вида. Тем не менее, если $H$ --- пустой граф, то полученная оценка точна.
    
    \begin{theorem}
        \label{t1_2}
        Пусть $m$ --- произвольное натуральное число, $F= P_3\sqcup mK_1$.
        Тогда $D[F]=v(F)-1$.
    \end{theorem}
    
    Заметим, что, во-первых, граф из теоремы~\ref{t1_2} является дизъюнктным объединением двудольных графов, а, во-вторых, при удалении из него хотя бы одного ребра, получается граф, являющийся реберным дополнением к полному многодольному (с одной долей размера 2 и остальными размера 1). Поэтому, в некотором смысле, граф из теоремы~\ref{t1_2} является наиболее простым дизюънктным объединением {\it различных} двудольных графов. В случае же, если двудольные графы изоморфны, утверждение теоремы сразу становится неверным, а верен тривиальный результат $D[F]=v(F)$. Более того, это верно даже для объединения одинаковых {\it многодольных} графов.
    
    Пусть $n_1 \leq \ldots \leq n_k$ --- натуральные числа. Рассмотрим $F:=mK_{n_1, \ldots, n_k}$ --- дизъюнктное объединение $m$ графов изоморфных полному многодольному графу $K_{n_1, \ldots, n_k}$ на $k$ долях, размеры которых равны $n_1, \ldots, n_k$.
    Обозначим $\ell := m \sum\limits_{i=1}^{k} n_k$ число вершин в этом графе.
    
    \begin{theorem}
        \label{t2}
        Выполнено равенство $D[F] = \ell$.
    \end{theorem}
    
    Заметим, наконец, что теорема~\ref{t1_1} позволяет лишь строить примеры графов $F$, для которых $D[F]\leq v(F)-1$, но мы до сих пор не знаем, существуют ли графы $F$, для которых $D[F]\leq v(F)-2$. Как было сказано во введении, для любого графа $F$ на не более $4$ вершинах $D[F]\geq v(F)-1$. Мы доказали, что то же самое верно и для графов на 5 вершинах, т.е. если $D[F]\leq v(F)-2$, то $v(F)\geq 6$.
    
    \begin{theorem}
    \label{t3}
    Для любого графа $F$ на $5$ вершинах выполнено неравенство $D[F]\geq 4$.
    \end{theorem}
    
    \section{Игра Эренфойхта--Фраиссе}
    \label{ehren}
    
    Основным средством в доказательстве теорем~\ref{t1_2},~\ref{t2},~\ref{t3} служит теорема А. Эренфойхта, доказанная в 1960 году~\cite{Ehren}. В данном разделе мы сформулируем ее частный случай для графов. Прежде всего определим игру Эренфойхта--Фраиссе на двух графах $G,H$ с количеством раундов, равным $r$ (см., например, \cite{Veresh,Survey,Libkin2004,VerbitskyZ17,Janson_RG}). В каждом раунде Новатор выбирает ровно одну вершину в любом из графов, отличную от уже выбранных (если такая существует --- иначе игра заканчивается). Консерватор следом в том же раунде должен выбрать вершину в {\it другом} графе, тоже отличную от уже выбранных. Если Консерватор не может этого сделать, то проигрывает. Если в каждом из графов хотя бы $r$ вершин (или одинаковое число вершин), то к концу игры выбраны вершины $x_1,\ldots,x_r$ и $y_1,\ldots,y_r$ в графах $G$ и $H$ соответственно (если в графах менее $r$ вершин, то выбрано меньшее число вершин, но это лишь меняет обозначение числа вершин и не влияет на дальнейшее описание, поэтому мы будем считать, что выбрано по $r$ вершин в каждом графе). Консерватор побеждает тогда и только тогда, когда отображение, переводящее $x_i$ в $y_i$, является изоморфизмом графов индуцированных на множества $\{x_1,\ldots,x_r\}$ и $\{y_1,\ldots,y_r\}$ (т.е. $x_i\sim x_j$ в $G$ тогда и только тогда, когда $y_i\sim y_j$ в $H$).

\begin{theorem} [А. Эренфойхт, 1960 \cite{Ehren}]
Для любых двух графов $G,H$ и любого $r\in\mathbb{N}$ Консерватор
имеет выигрышную стратегию в игре Эренфойхта--Фраиссе на графах $G,H$ в $r$ раундах тогда и только тогда, когда для любого предложения $\varphi\in\mathcal{F}$ кванторной глубины не более $r$ либо на $G$ и $H$ предложение $\varphi$ истинно, либо на обоих графах ложно.
\label{ehren_th}
\end{theorem}

Теперь поясним, как эту теорему можно применять для получения нижних оценок величины $D[F]$. Предположим, что граф $G$ содержит индуцированный подграф, изоморфный $F$, а $H$ не содержит. Пусть, кроме того, Консерватор побеждает в игре на графах $G$ и $H$ в $r$ раундах. Тогда в силу теоремы~\ref{ehren_th} $D[F]\geq r+1$. Действительно, если бы существовала формула глубины $r$, выражающая свойство $\mathcal{C}[F]$, то тогда бы она различала графы $G$ и $H$, а это противоречило бы теореме.

    \section{Графы, для которых существует нетривиальная запись}
    \label{proof1}
    
    \subsection{Доказательство теоремы~\ref{t1_1}}

    Пронумеруем вершины графа $H$: $h_1, \ldots, h_m$.
    Составим формулу глубины $m + 3$, выражающую существование подграфа, изоморфного $F$. Для этого рассмотрим следующие вспомогательные предикаты. 
    
    Во-первых, предикат
    $$
        P_H(x_1, \ldots, x_m) := \left[\bigwedge\limits_{i<j:\,h_i\sim h_j} (x_i\sim x_j)\right]\wedge\left[\bigwedge_{i<j:\,h_I\nsim h_j}(x_i\nsim x_j \land x_i \neq x_j)\right]
    $$
    является истинным тогда и только тогда, когда индуцированный подграф на вершинах $x_1, \ldots, x_m$ изоморфен графу $H$, причём отображение, переводящее $x_i$ в $h_i$ является изоморфизмом.
    
    Во-вторых, предикат
    $$
        P_0(x_1, \ldots, x_{m + 3}) := P_H(x_4, \ldots, x_{m + 3}) \land \left[\bigwedge\limits_{i = 1}^{3} \bigwedge\limits_{j = i + 1}^{m + 3} (x_i \neq x_j \land x_i \nsim x_j)\right]
    $$
    является истинным тогда и только тогда, когда вершины $x_1,x_2,x_3$ не смежны ни друг с другом ни с $x_4, \ldots, x_{m+3}$, а последние индуцируют $H$, причем $x_{3+i}\to h_i$ --- изоморфизм.
    
    Кроме того, предикат
    $$
        P_1(x_1, \ldots, x_{m + 3}) := P_H(x_4, \ldots, x_{m + 3}) \land x_1 \sim x_2\land\left[\bigwedge_{j=3}^{m+3}(x_1 \neq x_j \land x_1 \nsim x_j)\right] \land \left[\bigwedge\limits_{i = 2,3} \bigwedge\limits_{j = i + 1}^{m + 3} (x_i \neq x_j \land x_i \nsim x_j)\right]
    $$
    является истинным в той же ситуации, что и $P_0$, за исключением того, что вершины $x_1,x_2$ должны быть смежны.
    
    А для истинности предиката
    $$
        P_2(x_1, \ldots, x_{m + 3}) := P_H(x_4, \ldots, x_{m + 3}) \land x_1 \sim x_2 \land x_2 \sim x_3 \land x_1\neq x_3\land x_1\nsim x_3\land \left[\bigwedge_{i=1}^3\bigwedge_{j=4}^{m+3}(x_i \neq x_j \land x_i \nsim x_j)\right]
    $$
    должны быть смежны $x_1$ с $x_2$ и $x_2$ с $x_3$.\\
    
    Тогда искомая формула имеет вид
    \begin{multline*}
        \varphi=\E x_5 \ldots \E x_{m + 4}
         \E x_1\quad \biggl(\biggl[\E x_3 \E x_4\,\,\, P_0(x_1, x_3, x_4, x_5, \ldots, x_{m + 4})\biggr] \land \\
        \biggl[\E x_2\,\,\, \biggl([\E x_4 \, P_1(x_1, x_2, x_4, x_5, \ldots, x_{m + 4})] \land [\E x_3 \, P_2(x_1, x_2, x_3, x_5, \ldots, x_{m + 4})]\biggr)\biggr] \biggr).
    \end{multline*}
    
    Действительно, пусть сперва $G\in\mathcal{C}[F]$. Пронумеруем вершины индуцированного подграфа $\tilde F$, изоморфного $F$, в $G$ числами от $1$ до $m + 4$ таким образом, что $1\sim 2$, $2\sim 3$, $4$ --- изолированная вершина в $\tilde F$, и отображение $h_i\to i+4$ --- изоморфизм графов $H$ и $F_0$. Тогда, очевидно, истинность формулы $\varphi$ следует из того, что становятся истинны все предикаты в ней при подстановке $x_i = i$.\\
    
    С другой стороны, предположим, что $G\models\varphi$. Укажем такую нумерацию    некоторых вершин $G$ числами от $1$ до $m + 4$, что подграф на занумерованных вершинах изоморфен $F$.
    
    Пусть истинность формулы обеспечивается истинностью предикатов $P_0$, $P_1$, $P_2$ на наборах вершин $(x_1,y_1,y_2,x_5,\ldots,x_{m+4})$, $(x_1,x_2,x_4,x_5,\ldots,x_{m+4})$, $(x_1,x_2,x_3,x_5,\ldots,x_{m+4})$ соответственно. Пронумеруем вершины $x_5,\ldots,x_{m+4}$ числами $5, \ldots, m + 4$ соответственно.
    
    Заметим, что вершина $x_2$ не совпадает ни с $y_1$, ни с $y_2$, так как она смежна с $x_1$. Возможны 3 случая:
    \begin{enumerate}
        \item $x_2$ смежна ровно с одной вершиной из $y_1, y_2$. Не умаляя общности, пусть $y_1\sim x_2$. Тогда оставшиеся четыре номера $1,2,3,4$ получат вершины $x_1,x_2,y_1,y_2$ соответственно.

        \item $x_2$ не смежна ни с одной вершиной из $y_1, y_2$. Тогда вершина $x_3$ отлична как от $y_1$, так и от $y_2$.  Если $x_3$ не смежна с вершиной  $y_i$ для некоторого $i \in \{ 1, 2 \}$, то присвоим номера $1,2,3,4$ вершинам $x_1,x_2,x_3, y_i$ соответственно. Иначе ($x_3 \sim y_1$, $x_3 \sim y_2$) эти номера получат вершины $y_1,x_3,y_2,x_1$.
        
        \item $x_2\sim y_1$, $x_2\sim y_2$. Тогда вершина $x_4$ отлична как от $y_1$, так и от $y_2$.  Если для некоторого $i \in \{ 1, 2\}$ $x_4 \nsim y_i$, то номера $1,2,3,4$ получат вершины $x_1,x_2,y_i,x_4$ соответственно.
        Иначе ($y_1 \sim x_4$, $y_2 \sim x_4$) --- вершины $y_1,x_4,y_2,x_1$ соответственно.
    \end{enumerate}
            
    \subsection{Доказательство теоремы~\ref{t1_2}}

        В силу теоремы~\ref{t1_1} достаточно показать, что $D[F] \geq m + 2$. Для этого достаточно построить пару графов $G, G'$ таких, что $G\in\mathcal{C}[F]$,  $G'\notin\mathcal{C}[F]$ и у Консерватора есть выигрышная стратегия в игре Эренфойхта--Фраиссе на графах $G, G'$ в $m + 1$ раунде.\\
        
        Обозначим через $G_m$ граф с множеством вершин $\{1,\ldots,2m + 4\}$, у которого множество ребер описывается следующим образом: 
        $$
         1 \sim 2,\quad 2 \sim 3, \quad 3 \sim 4,
        $$
        $$
         2i + 3\sim 2i+4,\quad i\in\{1,\ldots,m\}.
        $$
        Иными словами, $G_m=P_4\sqcup mP_2$. 
        
        Обозначим $\mathcal{A}_i=\{2i+3,2i+4\}$, $i\in\{1,\ldots,m\}$, $\mathcal{B}=\{1,2,3,4\}$.
        
        Будем рассматривать $G := G_m$ и $G' := G_{m - 1}$. Нетрудно понять, что действительно $G\in\mathcal{C}[F]$,  $G'\notin\mathcal{C}[F]$. Осталось доказать, что у Консерватора есть выигрышная стратегия в игре Эренфойхта--Фраиссе на графах $G, G'$ в $m + 1$ раунде.\\

        Сначала опишем эту стратегию для первых $m - 1$ раундов.
        
        В {\bf первом раунде} Новатор выбирает либо вершину из $\mathcal{A}_i$ для некоторого $i\in\{1,\ldots,m\}$, либо вершину из $\mathcal{B}$. В первом случае Консерватор выбирает в другом графе вершину из $\mathcal{A}_1$. Во втором случае Консерватор выбирает вершину, {\it равную} вершине, выбранной Новатором.
        
        Пусть к началу {\bf раунда $r\in\{2,\ldots,m-1\}$} выбраны вершины $x_1,\ldots,x_{r-1}$ в $G$ и вершины $x_1',\ldots,x_{r-1}'$  в $G'$. Пусть, кроме того, выполнены следующие условия $(r-1).1$ и $(r-1).2$, где
        
        \begin{enumerate}
            \item[$\mu.1$] Для любого $i\in\{1,\ldots,\mu\}$
            $$
            (x_i\in\mathcal{B})\Leftrightarrow(x_i'=x_i).
            $$
            \item[$\mu.2$] Для любых различных $i,j\in\{1,\ldots,\mu\}$ 
            $$
            (\exists s\quad x_i,x_j\in\mathcal{A}_s)\Leftrightarrow
            (\exists s\quad x_i',x_j'\in\mathcal{A}_s).
            $$
        \end{enumerate}
        
        Заметим сразу, что выполнение этих условий гарантирует победу Консерватора в $\mu$ раундах.
        
        Очевидно, какую бы вершину ($x_r$ в $G$ или $x_r'$ в $G'$) ни выбрал бы Новатор, Консерватор сможет выбрать такую вершину ($x_r'$ в $G'$ или $x_r$ в $G$), что будут выполнены условия $r.1$, $r.2$.\\
        
        Итак, пусть {\bf сыграно $m-1$ раундов}. Если хотя бы в двух из них была выбрана вершина из $\mathcal{B}$, то, очевидно, Консерватор сможет в двух последних раундах играть так, что будут выполнены условия $(m+1).1$, $(m+1).2$ и, тем самым, победить.
        
        Если ровно в одном раунде была выбрана вершина из $\mathcal{B}$, то Консерватор сможет в раунде $m$ выбрать такую вершину, что выполнены условия $m.1$, $m.2$. Если в раунде $m$ Новатор выбрал вершину из $\mathcal{B}$, то Консерватор сможет и в раунде $m+1$ выбрать такую вершину, что выполнены условия $(m+1).1$, $(m+1).2$. Поэтому будем предполагать, что в раунде $m$ Новатор выбрал вершину вне $\mathcal{B}$. Кроме того, будем предполагать, что Новатор выбрал вершину в $G$ (иначе, опять же, в раундах $m$ и $m+1$ Консерватор сможет выбрать такие вершины, что выполнены условия $(m+1).1$, $(m+1).2$). Тогда в раунде $m+1$ Новатору не выгодно выбирать вершину из $\mathcal{B}$, так как в этом случае Консерватор сможет просто выбрать ту же самую вершину. Не выгодно Новатору и выбирать вершину из некоторого $\mathcal{A}_s$, в котором уже выбрана некоторая вершина $x_i$, так как в этом случае Консерватор просто выберет вершину из того $\mathcal{A}_{s'}$, в котором лежит $x_i'$. Если же Новатор выберет вершину из единственного множества $\mathcal{A}_s$, в котором еще не выбрано ни одной вершины $x_i$, то Консерватор выберет вершину из $\mathcal{B}$, которая не смежна с единственной вершиной, выбранной в $\mathcal{B}$, и победит.
        
        Пусть, наконец, все $x_i$, $i\in\{1,\ldots,m-1\}$, лежат вне $\mathcal{B}$. Если в раунде $m$ Новатор выберет вершину из $\mathcal{B}$, то Консерватор выбирает такую же вершину, и, очевидно, сможет победить в раунде $m+1$, так как в графе $G'$ существует по крайней мере еще одна вершина (из $\mathcal{B}$), которая не смежна ни с одной из $x_1',\ldots,x_m'$. Если в раунде $m$ Новатор выберет вершину (скажем, в $G$) из некоторого $\mathcal{A}_s$, в котором уже выбрана некоторая вершина $x_i$, $i\in\{1,\ldots,m-1\}$, то Консерватор должен выбрать вершину $x_m'$ из того $\mathcal{A}_{s'}$, в котором лежит $x_i'$. Очевидно, в раунде $m+1$ Консерватор победит, так как, если даже Новатор выберет вершину из единственного оставшегося нетронутым $\mathcal{A}_s$, то Консерватор выберет произвольную вершину из $\mathcal{B}$. Наконец, если в раунде $m$ Новатор выберет вершину (скажем, в $G$) из $\mathcal{A}_s$, которому не принадлежит ни одна вершина $x_i$, $i\in\{1,\ldots,m-1\}$, то Консерватор выберет вершину $x_m'=1\in\mathcal{B}$. Так как в каждом графе
        
        \begin{itemize}
        
        \item у каждой выбранной вершины, не смежной с другими выбранными, есть ровно один сосед, 
        
        \item у каждой выбранный вершины, имеющей соседа среди выбранных, нет других соседей, 
        
        \item ни у какой пары выбранных вершин нет общего соседа, 
        
        \item есть вершина, не смежная ни с одной из выбранных,
        
        \end{itemize}
        
        то Консерватор сможет победить в последнем раунде.

    \section{Графы, для которых не существует нетривиальной записи}
    \label{proof2}
    
    Для доказательства теоремы~\ref{t2}  достаточно построить два графа $G, G'$ такие, что $G\in\mathcal{C}[F]$, $G'\notin\mathcal{C}[F]$ и у Консерватора есть выигрышная стратегия в игре Эренфойхта--Фраиссе на графах $G,G'$ в $\ell - 1$ раунде.
    
    Определим {\it почти многодольный граф} $G^m_{\ell, k}$ как граф на множестве вершин $V := \{  (i, j, h) \mid 1 \leq i \leq\ell, 1 \leq j \leq k, 1 \leq h \leq m \}$,
    и с множеством рёбер 
    $$ 
    E := \{ ( (i_1, j_1, h), (i_2, j_2, h)) \mid \forall i_1, j_1, i_2, j_2, h: i_1 \neq i_2, j_1 \neq j_2 \}  \cup 
    \{ ( (i, j_1, h_1), (i, j_2, h_2)) \mid \forall i, j_1, j_2, h_1, h_2: h_1 \neq h_2\}
    $$
    
    Для вершины $(i, j, h)$ будем называть $i$ --- {\it классом}, $j$ --- {\it долей}, а $h$ --- {\it компонентой} этой вершины, при этом будем говорить, что {\it вершины $(i_1,j_1,h_1)$, $(i_2,j_2,h_2)$  находятся в одной доле}, если номера их долей и компонент совпадают (т.е. $j_1=j_2$, $h_1=h_2$).
    
    Заметим, что если инвертировать все рёбра в этом графе между вершинами с одинаковым классом, то получится дизъюнктное объединение $m$ полных $k$-дольных графов $mK_{\ell, \ldots, \ell}$.
    
    Положим $G=G^m_{\ell, k}$, $G'=G^m_{\ell-1, k}$. Докажем, что это нужные нам графы. Разобьем доказательство на несколько утверждений.
    
    \begin{proposition}
        Для любых $\ell \geq 2, m \geq 1, k \geq 1$ у Консерватора есть выигрышная стратегия в игре Эренфойхта--Фраиссе на графах $G$ и $G'$ в $\ell - 1$ раунде.
        \label{prop:Dupwin}
    \end{proposition}
    \begin{proof}
        Будем строить функцию $f: \range{1}{\ell} \rightarrow  \range{1}{\ell - 1}$, переводящую классы вершин графа $G$ в классы вершин графа $G'$. В начале все значения функции не определены. 
        
        Пусть Новатор в некотором раунде выбрал вершину $(i, j, h)$ в одном из графов $G, G'$ такую, что в этом графе нет другой выбранной вершины с этим же классом. Тогда пусть $i'$ --- такой класс вершины в другом графе, что в нет тоже нет выбранной вершины с классом $i'$. Такой всегда найдётся, так как классов к графах $G, G'$ хотя бы  $\ell - 1$, а раундов всего $\ell - 1$. Консерватор выберет вершину $(i', j, h)$. Если Новатор делал ход в графе $G$, то определим $f(i) = i'$, и $f(i') = i$ иначе.
        
        Если же Новатор в этом раунде выбрал вершину $(i, j, h)$ такую, что в том же графе уже есть выбранные вершины с таким же классом, то либо корректно определена $i':=f(i)$ (при условии что Новатор выбирал вершину в графе $G$), либо корректно определена $i' := f^{-1}(i)$ (при условии что Новатор выбирал вершину в графе $G'$). Тогда Консерватор выбирает вершину $(i', j, h)$ в другом графе.
        
        Так как при описанной стратегии совпадение номеров классов/долей/компонент выбранных вершин в одном графе равносильно совпадению номеров классов/долей/компонент соответствующих выбранных вершин в другом графе, то индуцированные подграфы на выбранных вершинах по окончанию игры будут изоморфны.

    \end{proof}

    \begin{proposition}
    \label{prop:ind}
        Граф $G$ содержит индуцированный подграф, изоморфный $F$.
    \end{proposition}
    \begin{proof}
        Пронумеруем все вершины и компоненты графа $mK_{n_1, \ldots, n_k}$, а также доли каждой компоненты числами от $1$ до $k$. Тогда вершине с номером $i$ из компоненты $h$ этого графа из доли с номером $j$ сопоставим вершину $(i, j, h)$ графа $G$. Это и будет искомое вложение.
    \end{proof}

    Условимся далее, что везде, где речь будет идти о {\it доле образа вершины} из $F$ при вложении $F$ в некоторый почти многодольный граф $A$, будет иметься в виду доля графа $A$, в которой лежит этот образ. Соответственно, когда будет упоминаться доля вершины из $F$, будет иметься в виду доля графа $F$.

    \begin{proposition}
        \label{prop:unipart}
        Пусть $n_k>1$, $k\geq 2$. Если граф $G'$ содержит индуцированный подграф, изоморфный $F$, то не существует двух вершин $v$, $w$, лежащих в разных компонентах графа $F$ таких, что их образы лежат в одной части графа $G'$ и в разных долях графа $G'$.
    \end{proposition}
    \begin{proof}
        Предположим, что такая пара вершин $v, w$ нашлась. Обозначим через $v', w'$ их образы.
        
        Так как $w$, $v$ лежат в разных компонентах, $w' \nsim v'$. Значит, классы $v'$ и $w'$ совпадают.
        
        Выберем $w_0$ такую, что $w_0, w$ лежат в одной компоненте связности, но разных долях. Тогда $w_0 \sim w$. Рассмотрим её образ 
        $w'_0$. 
        
        Если $w'_0$ лежит в отличной от $v'$ части, то тогда так как $w'_0 \sim w'$, то классы $w'_0, w'$ совпадают.
        Но тогда классы $v', w'_0$ совпадают и $v' \sim w'_0$. Противоречие.
        
        Значит, $w'_0$ лежит в одной части с $v'$. Если доли $w'_0$ и $v'$ различны, то так как $v' \nsim w'_0$, 
        классы $w'_0, v$ совпадают. Но тогда классы $w', w'_0$ совпадают и либо $w' = w'_0$, либо $w' \nsim w'_0$. Противоречие.
        
        Значит, $w'_0$ лежит в одной доли с $v'$. Меняя в этом рассуждении $w'$ и $v'$ местами, получаем, что в $F$
        существует $v_0$ с образом $v'_0$, лежащая в одной компоненте с $v$, но в разных с ней долях, и $v'_0$ лежит
        в одной части и одной компоненте с $w'$. Тогда так как $v'_0$, $w'_0$ лежат в разных долях одной части
        и не соединены ребром, то их класс совпадает.
        
        Так как $n_k > 1$, то в одной компоненте связности $F$ есть хотя бы три вершины.
        Значит, существует вершина $v_1$ из одной компоненты с $v, v_0$. При этом, так как $v, v_0$ из разных долей,
        $v_1$ соединена хотя бы с одной из них. Пусть, не умаляя общности, это $v$. Пусть $v'_1$ --- образ $v_1$.
        
        Так как $w'_0, w'$ имеют разные классы, и в долях, в которых лежат $w'_0, w$ эти классы
        заняты отличными от $v'_1$ вершинами, то если $v'_1$ принадлежит одной части с ними, то она смежна и с $w'$, чего не может быть.
        
        Значит, $v'_1$ лежит в другой части. Но, так как $v'_1 \sim v'$, она имеет один класс с $v'$, и следовательно, с $w'_0$. Значит, $v'_1 \sim w'_0$. Противоречие.
        
    \end{proof}

    \begin{proposition}
        \label{prop:noind}
        Пусть $n_k>1$, $k\geq 2$. Тогда граф $G'$ не содержит индуцированный подграф, изоморфный $F$.
    \end{proposition}
    \begin{proof}        
        Предположим противное: $G'\in\mathcal{C}[F]$, и рассмотрим некоторое вложение $F$ (как индуцированного подграфа) в $G'$.\\
        
        Докажем сперва, что в любой компоненте $F$ найдутся две вершины, образы которых лежат в одной части, но в разных долях графа $G'$.\\
        
        Рассмотрим произвольную компоненту $F$.
        
        Так как $n_k > 1$, то найдутся две вершины $v_1, v_2$, лежащие в одной доле в этой компоненте, и $v_0$, лежащая в отдельной доле от $v_1, v_2$ в той же компоненте. Тогда $v_1 \nsim v_2$, $v_1 \sim v_0$, $v_2 \sim v_0$.
        
        Предположим, что все образы $v_0, v_1, v_2$ --- $v'_0, v'_1, v'_2$ соответственно --- лежат в разных частях графа $G'$. Тогда классы $v_0, v_1$  и $v_0, v_2$ должны совпадать. Но тогда $v_1 \sim v_2$. Противоречие.
        
        Предположим, что какие-то два образа лежат в одной части, а третий --- в другой. Если они соединены ребром, то они обязаны лежать в разных долях, и нужное доказано. Иначе образы из одной части --- это $v'_1, v'_2$. Тогда классы $v'_0, v'_1$ и $v'_0, v'_2$ совпадают и, следовательно, $v'_1, v'_2$ лежат в разных долях, что и требовалось.
        
        Наконец, если все три образа лежат в одной части, то $v'_0, v'_1$ должны лежать в разных долях.\\
        
        Значит, для любой компоненты $F$ можно выделить хотя бы одну часть графа $G'$, в двух разных долях которой лежат образы вершин этой компоненты.
        Но тогда, согласно утверждению~\ref{prop:unipart}, образ любой вершины из любой другой компоненты $F$, лежащей в этой части, должен лежать сразу в двух долях,
        чего не может быть. Значит, в этой части могут лежать только образы этой компоненты.
        
        Так как компонент столько же, сколько частей, каждая компонента $F$ вложена целиком в одну часть $G'$.\\
        
        Рассмотрим произвольную компоненту $F$. Так как вершины из разных долей этой компоненты соединены ребром, то их образы не могут лежать в одной доле. Так как долей в рассматриваемой компоненте $F$ столько же, сколько в одной части в $G'$, то образ каждой доли лежит ровно в одной доле.
        
        Предположим теперь, что какие-то два образа $v'_1, v'_2$ вершин $v_1, v_2$ из рассматриваемой компоненты графа $F$ имеют один класс. Тогда $v'_1 \nsim v'_2$ и лежат в разных долях. Значит $v_1 \nsim v_2$, откуда $v_1, v_2$ лежат в одной доле $H$. Но тогда их образы тоже должны лежать в одной доле --- противоречие. Таким образом, образы всех вершин рассматриваемой компоненты графа $F$ имеют попарно различные классы.\\
        
        Наконец, образы вершин из разных частей не могут иметь совпадающего класса, так как иначе, поскольку они лежат в разных частях, они были бы смежны.\\
        
        Следовательно, образы различных вершин из $F$ имеют разные классы в $G'$. Но классов всего $\ell-1$ --- меньше, чем вершин. Противоречие. Следовательно, $G'\notin\mathcal{C}[F]$.
    \end{proof}

        В силу утверждений~\ref{prop:Dupwin},~\ref{prop:ind}~и~\ref{prop:noind} при $k\geq 2$ и $n_k > 1$ теорема доказана.
        
        При $k=1$ $F=mn_1 K_1$, т.е. является пустым графом. В этом случае $D[F]=D[\overline{F}]=\ell$~\cite{VerbitskyZ17}.
        
        Наконец, при $k\geq 2$ и $n_k = 1$ граф $F$ равен дизъюнктному объединению полных графов: $F=mK_k$. Так как $D[F]=D[\overline{F}]$ и $\overline{F}$ --- полный многодольный граф, то при $m>1$ $D[F]=\ell$ по уже доказанному (при $k \geq 2$ $\overline{F}$ попадает под уже рассмотренный случай, если $m > 1$). А при $m = 1$ $F$ --- полный граф, для которого также верно $D[F] = \ell$.

    \section{Графы на 5 вершинах}
    \label{proof3}
    
    Заметим сперва, что для любого графа $F$ и реберного дополнения $\overline{F}$ к нему выполнено $D[F]=D[\overline{F}]$. Поэтому в силу~(\ref{general_lower})
    $$
     D[F]\geq\frac{\max\{e(F),{v(F)\choose 2}-e(F)\}}{v(F)}+2.
    $$
    В частности, для любого графа $F$ на $5$ вершинах и числом ребер, отличном от $5$, выполнено $D[F]\geq\frac{6}{5}+2$, а значит $D[F]\geq 4$. Осталось проверить справедливость этого неравенства для всех графов с ровно $5$ ребрами. Принимая во внимание то, что нам не нужно рассматривать реберные дополнения, мы можем ограничиться следующими четырьмя графами:
    
    \begin{enumerate}
    
    \item простой цикл на 5 вершинах $C_5$;
    
    \item граф $G_{4,1}$, равный объдинению простого цикла на 4 вершинах и ребра, исходящего из одной из вершин цикла к пятой вершине;
    
    \item граф $K_1\sqcup K_{2,1,1}$, равный дизъюнктному объединению вершины и diamond-graph;
    
    \item граф $G_{3,1,1}$, равный объединению треугольника и двух непересекающихся по вершинам ребер, одно из которых соединяет одну вершину треугольника с четвертой вершиной графа, а второе соединяет другую вершину треугольника с пятой вершиной графа.
    
    \end{enumerate}
    
    \subsection{Граф $C_5$}
    
    Итак, докажем сперва, что $D[C_5]\geq 4$. Для этого рассмотрим графы $G=C_5\sqcup C_6$ и $H=C_6\sqcup C_6$. Очевидно, $G\in\mathcal{C}[C_5]$, а $H\notin\mathcal{C}[C_5]$. Осталось доказать, что Консерватор выигрывает в игре Эренфойхта--Фраиссе на графах $G,H$ в 4 раундах. 
    
    Если Новатор в первом раунде отметил вершину в одном из $C_6$, то дальнейшая стратегия Консерватора очевидна. Если же Новатор в первом раунде выбрал некоторую вершину $C_5$, то Консерватор в том же раунде должен выбрать произвольную вершину графа $H$. 
    
    Очевидно, Новатору не выгодно во втором раунде выбирать вершину в одной из двух компонент, ни одна из вершин которых не выбрана в первом раунде. Предположим сначала, что Новатор выбрал еще одну вершину $C_5$. Тогда Консерватор выбирает в ``своем'' $C_6$ вершину, лежащую от выбранной на расстоянии, равном расстоянию между выбранными вершинами в $C_5$ (под расстоянием подразумевается количество ребер в наикратчайшем простом пути, соединяющем рассматрвиаемые вершины). Очевидно, Консерватор сможет победить в третьем раунде. 
    
    Пусть, наконец, во втором раунде Новатор выбрал вершину из $C_6$ в $H$, в котором Консерватором уже выбрана одна вершина. Тогда от предыдущей ситуации отличается лишь случай, в котором Новатор выбирает вершину, находящуюся на расстоянии 3 от уже выбранной. В этом случае Консерватор должен выбрать вершину в $C_6\subset G$. Так как у обеих пар вершин нет общей вершины, а все остальные варианты смежности третьей вершины с ними присутствуют, то в третьем раунде Консерватор сможет победить.
    
    \subsection{Граф $G_{4,1}$}
    
    На рис.~\ref{ris:image1} а) изображены графы $G$ и $H$,  полученные из $C_4$ и $C_5$ соответственно добавлением непересекающихся ребер к каждой из вершин циклов. Очевидно, $G$ содержит индуцированный подграф, изоморфный $F=G_{4,1}$, а $H$ --- нет. Докажем, что у Консерватора есть выигрышная стратегия в игре Эренфойхта--Фраиссе на графах $G,H$ в $3$ раундах.
    
    Для победы Консерватору достаточно просто ``повторять'' ходы Новатора, а именно обеспечить выполнение следующих двух условий:
    
    \begin{enumerate}
    
        \item[1)] выбирать вершины степени 1 тогда и только тогда, когда Новатор в том же раунде выбирает вершины степени 1, и выбирать вершины цикла тогда и только тогда, когда Новатор в том же раунде выбирает вершины цикла;
        
        \item[2)] во втором раунде выбирать вершину, находящуюся от первой на том же расстоянии, что и две соответствующие вершины в графе, выбранном во втором раунде Новатором.
        
    \end{enumerate} 
    
    Очевидно, Консерватор сможет в первых двух раундах следовать условию 1), а во втором раунде следовать условию 2). Тогда, как легко видеть, в третьем раунде он сможет победить.

\begin{figure}[h]
\begin{minipage}[h]{0.49\linewidth}
\center{\includegraphics[width=0.8\linewidth]{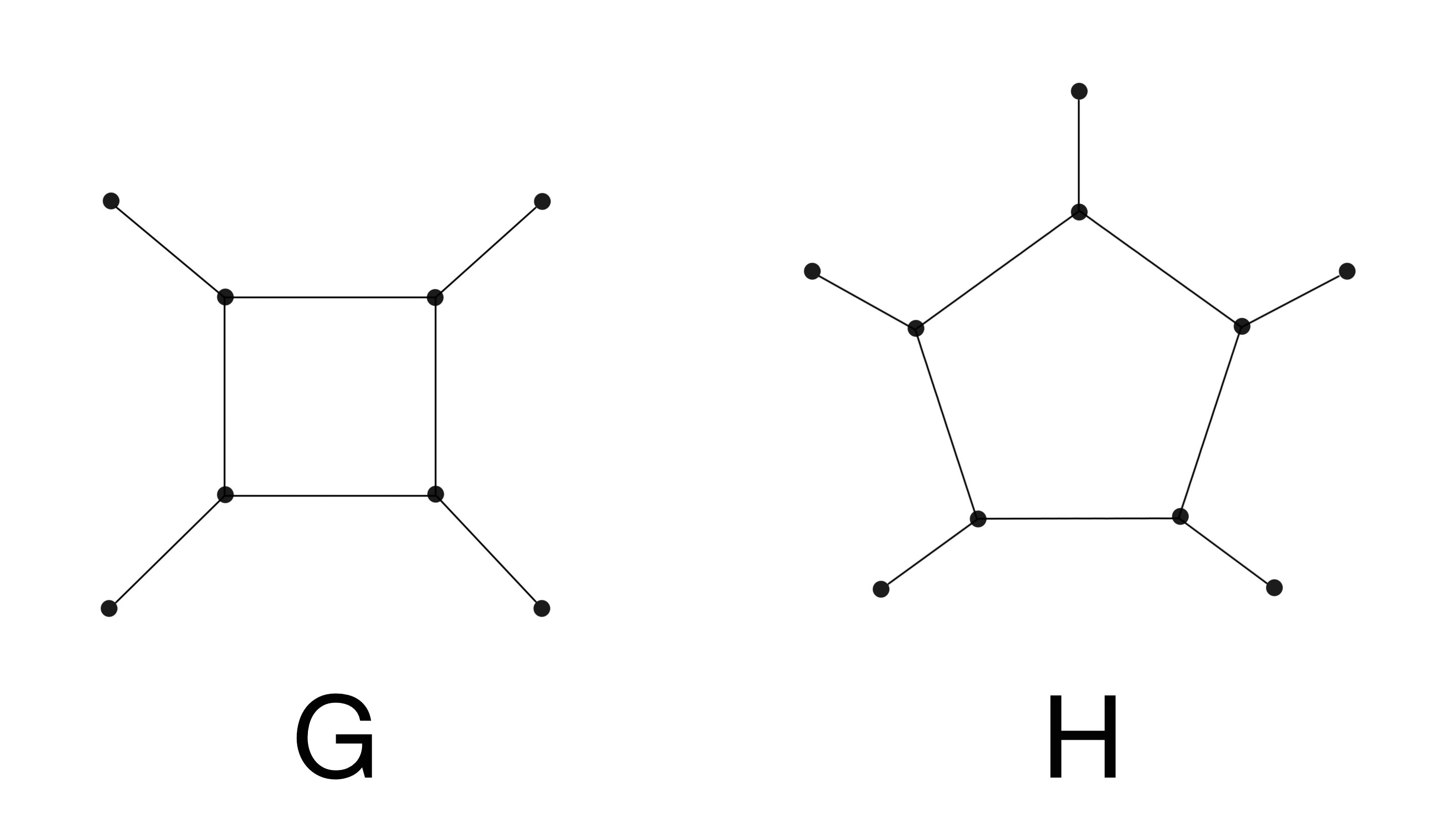} \\ а)}
\end{minipage}
\hfill
\begin{minipage}[h]{0.49\linewidth}
\center{\includegraphics[width=0.8\linewidth]{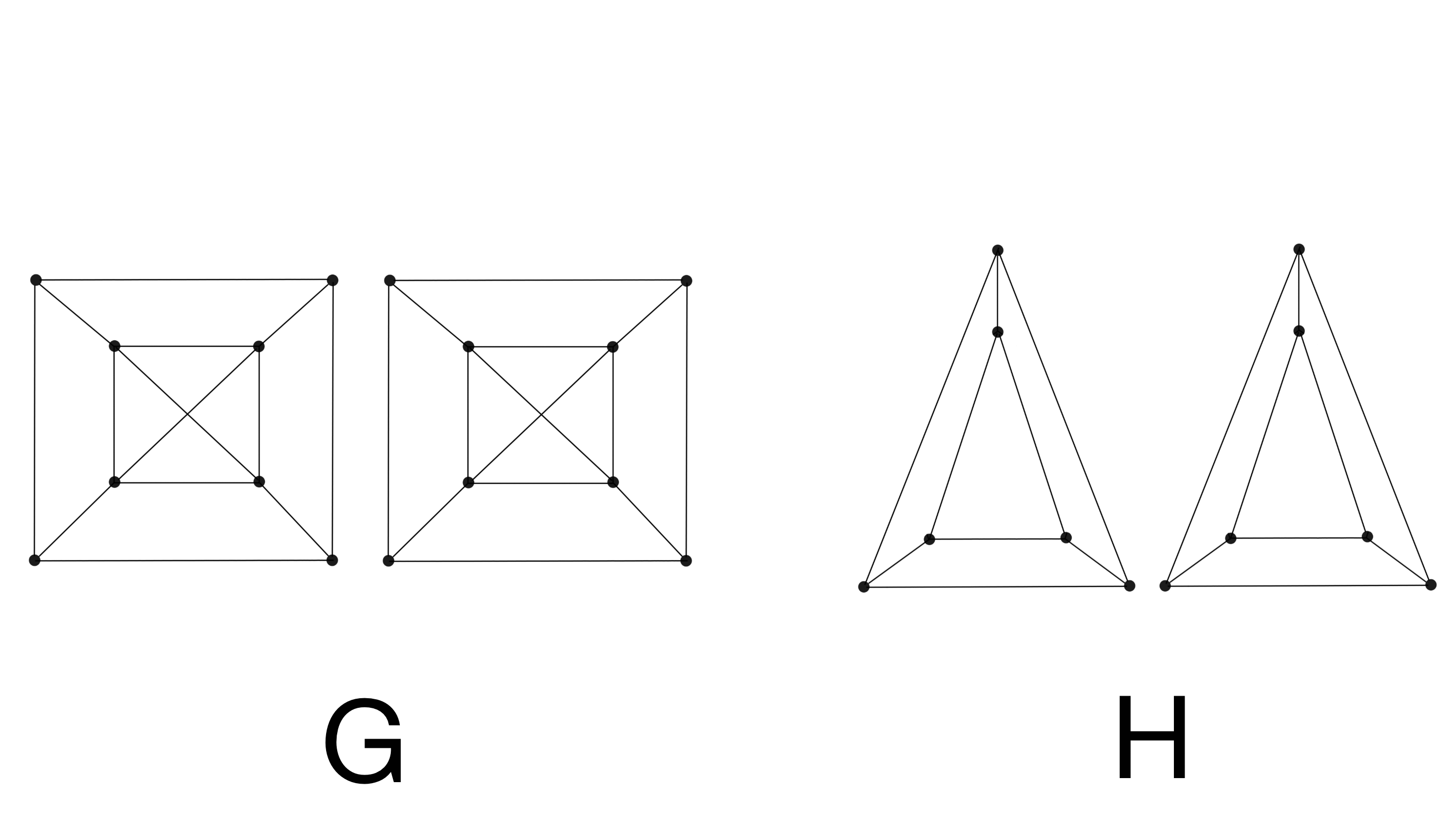} \\ б)}
\end{minipage}
\caption{Графы $G$ и $H$}
\label{ris:image1}
\end{figure}
    
    \subsection{Граф $K_1\sqcup K_{2,1,1}$}
    
        Заметим, что diamond-graph $K_{2,1,1}$ является полным многодольным, а значит, как мы установили в разделе~\ref{proof2}, существует такие графы $G,H$, что $G\in\mathcal{C}[K_{2,1,1}]$, $H\notin\mathcal{C}[K_{2,1,1}]$ и Консерватор выигрывает в игре Эренфойхта--Фраиссе на графах $G,H$ в 3 раундах. Осталось заметить, что отсюда очевидно вытекает, что, во-первых, $2G\in\mathcal{C}[K_1\sqcup K_{2,1,1}]$, $2H\notin\mathcal{C}[K_1\sqcup K_{2,1,1}]$, а, во-вторых, Консерватор выигрывает в игре Эренфойхта--Фраиссе на графах $2G,2H$ в 3 раундах.
    
    \subsection{Граф $G_{3,1,1}$}
    
        На рис.~\ref{ris:image1} б) изображены графы $G$ и $H$,  полученные из $K_4\sqcup C_4$ и $K_3\sqcup K_3$ соответственно добавлением ребер между ``соответсвующими'' вершинами компонент, а затем рассмотрением дизънктных объединений двух копий каждого из полученных графов. Очевидно, $G$ содержит индуцированный подграф, изоморфный $F=G_{3,1,1}$, а $H$ --- нет. Пусть граф $G^*$ получен из графа $G$ добавлением вершины $x^*$, смежной со всеми вершинами $G$, а $H^*$ --- из $H$ добавлением вершины $y^*$, смежной со всеми вершинами $H$. Разумеется, для этих графов верно то же самое: $G^*\in\mathcal{C}[F]$, $H^*\notin\mathcal{C}[F]$. Докажем, что у Консерватора есть выигрышная стратегия в игре Эренфойхта--Фраиссе на графах $G^*,H^*$ в $3$ раундах.\\

        Ниже мы будем говорить, что пара вершин $(u,v)$ некоторого графа $A$ обладает {\it свойством расширения}, если в графе $A$ найдутся четыре других вершины $z_1,z_2,z_3,z_4$, состоящие в разных отношениях смежности с $u,v$, т.е. $u\nsim z_1$, $v\nsim z_1$; $u\sim z_2$, $v\nsim z_2$; $u\nsim z_3$, $v\sim z_3$; $u\sim z_4$, $v\sim z_4$.  Разумеется, если в первых двух раундах игры выбраны вершины $x_1,x_2$ в $G$ и $y_1,y_2$ в $H$ (либо $x_1\sim x_2$, $y_1\sim y_2$, либо $x_1\nsim x_2$, $y_1\nsim y_2$), то если обе пары $(x_1,x_2)$ и $(y_1,y_2)$ обладают свойством расширения, то Консерватор сможет выиграть в третьем раунде.\\
        
        Опишем теперь выигрышную стратегию Консерватора.  Для победы Консерватору достаточно просто обеспечить выполнение следующих двух условий:
    
    \begin{enumerate}
    
        \item[1)] для каждого $i\in\{1,2\}$ в раунде $i$ в $G^*$ выбрана вершина $x^*$ тогда и только тогда, когда в том же раунде в $H^*$ выбрана вершина $y^*$;
        
        \item[2)] если в первых двух раундах не выбрана ни вершина $x^*$, ни вершина $y^*$ (т.е. выбраны вершины $x_1,x_2$ в $G$ и $y_1,y_2$ в $H$), то $x_1\sim x_2$ тогда и только тогда, когда $y_1\sim y_2$. 
        
    \end{enumerate} 
    
    Очевидно, Консерватор сможет в первых двух раундах следовать обоим условиям. Если при этом хотя бы в одном раунде выбраны вершины $x^*,y^*$, то Новатор в том раунде просто потерял ход, так как эти вершины смежны со всеми остальными, и дальнейшая стратегия Консерватора очевидна. Если же вершины $x^*,y^*$ не выбраны, то обе пары $(x_1,x_2)$ и $(y_1,y_2)$ обладают свойством расширения, а значит у Консерватора есть выигрышная стратегия в третьем раунде (действительно, в силу существования вершин $x^*,y^*$, а также в силу того, что оба графа $G$ и $H$ состоят из двух компонент связности, достаточно проверить, что у любых двух вершин в $G$ и в $H$ есть сосед только первой и сосед только второй, а это легко проверяется).

    \section{Открытые вопросы}
    \label{discussions}
    
    Хотя теорема \ref{t1_1} показывает, что существуют сколь угодно большие графы, для которых $D[F] < v(F)$, остаётся открытым вопрос, существуют ли графы, у которых $D[F] < v(F) - 1$. Вполне возможно, что существует граф $F$ на 6 вершинах, для которого $D[F]=4$ (аналогично рассуждению в разделе~\ref{proof3} из~(\ref{general_lower}) незамедлительно следует что графа на 6 вершинах с $D[F]<4$ не существует).
    
    Хочется высказать еще более смелое предположение о существовании такого $\alpha<1$ и графов $F$ сколь угодно большого размера, для которых $D[F]\leq\alpha v(F)$. С точки зрения вклада в задачу о временной сложности проверки $\mathcal{C}[F]$ хочется доказать, что $\alpha$ может быть меньше чем $\frac{\omega}{3}$.

    Наконец, возникает вопрос, можно ли с обобщить теорему \ref{t2}, используя похожую технику, для более широкого класса графов? Например, для некоторых дизъюнктных объединений полных многодольных {\it неизоморфных} графов
    или, скажем, для графов, состоящих из независимых множеств $I_1,\ldots,I_k$ размера хотя бы 2, таких, что между любыми двумя множествами смежности одинаковы (для любых различных $i,j\in\{1,\ldots,k\}$ либо каждая вершина $I_i$ смежна с каждой вершиной $I_j$, либо между $I_i$ и $I_j$ вообще нет ребер).

\end{document}